\documentclass[generic]{imsart}

\setattribute{journal}{name}{Submitted to Biometrika}

\RequirePackage[OT1]{fontenc}
\RequirePackage{amsthm,amsmath,booktabs}
\RequirePackage[authoryear]{natbib}
\RequirePackage[colorlinks,citecolor=blue,urlcolor=blue]{hyperref}

\usepackage{natbib}
\usepackage{booktabs}
\usepackage{amssymb,amsmath,subcaption}
  \numberwithin{equation}{section}
\usepackage{color, caption, geometry}
\usepackage{graphicx, epstopdf}
\usepackage{algorithm}
\usepackage[noend]{algpseudocode}

\makeatletter
\def\BState{\State\hskip-\ALG@thistlm}
\makeatother

\newtheorem{theorem}{Theorem}

\newtheorem{lemma}[theorem]{Lemma}

\newtheorem{remark}[theorem]{Remark}

\newcommand{\BALD}{\begin{aligned}}
\newcommand{\EALD}{\end{aligned}}
\newcommand{\BALDS}{\begin{aligned*}}
\newcommand{\EALDS}{\end{aligned*}}
\newcommand{\BCAS}{\begin{cases}}
\newcommand{\ECAS}{\end{cases}}
\newcommand{\BEAS}{\begin{eqnarray*}}
\newcommand{\EEAS}{\end{eqnarray*}}
\newcommand{\BEQ}{\begin{equation}}
\newcommand{\EEQ}{\end{equation}}
\newcommand{\BIT}{\begin{itemize}}
\newcommand{\EIT}{\end{itemize}}
\newcommand{\BMAT}{\begin{bmatrix}}
\newcommand{\EMAT}{\end{bmatrix}}
\newcommand{\BNUM}{\begin{enumerate}}
\newcommand{\ENUM}{\end{enumerate}}

\newcommand{\BA}{\begin{array}}
\newcommand{\EA}{\end{array}}

\newcommand{\real}{\mathbb{R}}

\DeclareMathOperator*{\argmin}{\arg\min}
\newcommand{\diag}{\mathop{\mathbf{diag}}}

\DeclareMathOperator{\sign}{sign}

\newcommand{\hsigma}{\widehat{\sigma}}

\newcommand{\Ee}{\mathbb{E}}
\newcommand{\Pp}{\mathbb{P}}
\newcommand{\Q}{\mathbb{Q}}
\newcommand{\M}{\mathbb{M}}

\newcommand{\hz}{\hat{z}}

\newcommand{\hbeta}{\hat{\beta}}

\newcommand{\bX}{X}

\providecommand{\argmin}{\mathop\mathrm{arg min}}

\providecommand{\diag}{\mathop\mathrm{diag}}

\providecommand{\sign}{\mathop\mathrm{sign}}

\def\P{\mathbb{P}} 

\def\bar{\overline}

\begin{document}

\begin{frontmatter}
\title{Selective inference with unknown variance via the square-root LASSO}
\runtitle{Selective inference with unknown variance}

\begin{aug}
\author{\fnms{Xiaoying} \snm{Tian}\ead[label=e1]{xtian@stanford.edu}},
\and
\author{\fnms{Joshua R.} \snm{Loftus}\ead[label=e2]{joftius@stanford.edu}}, 
\and
\author{\fnms{Jonathan E.} \snm{Taylor}\ead[label=e3]{jonathan.taylor@stanford.edu}}

\runauthor{Tian et al.}

\affiliation{Stanford University}

\address{Department of Statistics\\
Stanford University\\
Stanford, California\\
\printead{e1}, \\
\printead*{e2}}
\end{aug}

\begin{abstract}
There has been much recent work on inference after model selection when the noise level is known, 
however, $\sigma$ is rarely known in practice and its estimation is difficult in high-dimensional settings.
In this work we propose using the square-root LASSO (also known as the scaled LASSO) to perform selective inference for the coefficients and the noise level simultaneously. The square-root LASSO
has the property that choosing a reasonable tuning parameter is scale-free, namely it does not depend on the noise level
in the data.
We provide valid p-values and confidence intervals for the coefficients after selection,
and estimates for model specific variance. Our estimates perform better than other estimates of $\sigma^2$ in simulation. 
\end{abstract}

\begin{keyword}[class=AMS]
\kwd[Primary ]{62F03}
\kwd{62J07}
\kwd[; secondary ]{62E15}
\end{keyword}

\begin{keyword}
\kwd{lasso}
\kwd{square root lasso}
\kwd{confidence interval}
\kwd{hypothesis test}
\kwd{model selection}
\end{keyword}

\end{frontmatter}

\section{Introduction}
\label{sec:intro}

Selective inference differs from classical inference in regression.
Given $y \in \real^n, X \in \real^{n \times p}$ we first
choose a model by selecting some subset $E$ of the columns of $X$.
Denoting this model submatrix $X_E$, we proceed with the related
regression model
\begin{equation} 
y = X_E\beta_E + \epsilon, \epsilon \sim N(0, \sigma^2_EI),
\label{eq:model}
\end{equation}
and conduct the usual types of inference considered in regression
such as hypothesis tests and confidence intervals.

Most previous literature \citep{taylor2013tests,lee2013exact,spacings}
assumes that $\sigma^2_E$ is known. This is problematic for two reasons: first, it is
almost never known in practice; second, the noise level $\sigma_E$ as posited above
is specific to the model we choose. As we choose the variables $E$ with data, it is
not generally easy to get an independent estimate of $\sigma_E$. 
In this work  we propose a method that will treat $\sigma_E$ as one of the
parameters for inference and adjust for selection.

Our method is both valid in theory and practice. We illustrate the latter through
comparisons of estimates of $\sigma_E^2$ with \cite{sun2012scaled, reid2013study},
and FDR control and power with \cite{barber2016knockoff}.

\subsection{The Square-root LASSO and its tuning parameters}

The selection procedure we use is based on the square-root LASSO \cite{sqrt_lasso}, 
which in turn is known to be equivalent to the scaled LASSO \cite{sun2012scaled}. 
\begin{equation}
\label{eq:sqrt_lasso}
\hat{\beta}_{\lambda} = \argmin_{\beta \in \real^p}  \|y-X\beta\|_2 + \lambda \cdot \|\beta\|_1.
\end{equation}
The square-root LASSO is a modification of the LASSO \cite{tibshirani1996regression}:
\begin{equation}
\label{eq:lasso}
\tilde{\beta}_{\gamma} = \argmin_{\beta \in \real^p} \frac{1}{2} \|y-X\beta\|^2_2 + \gamma \cdot \|\beta\|_1.
\end{equation}

The first advantage of using square-root LASSO is the convenience in choosing $\lambda$. 
For the LASSO, a good choice of $\gamma$ depends on the noise variance $\sigma_E$, \cite{negahban_unified_2012} 
\begin{equation}
\label{eq:gamma}
\gamma = 2 \cdot \Ee (\|X^T\epsilon\|_{\infty}), \qquad \epsilon \sim N(0, \sigma^2_E I)
\end{equation}
In practice, we might consider some multiple other than 2. As $\lambda_1=\lambda_1(X,y)$, the first
knot on the solution path of \eqref{eq:lasso}, is equal to $\|X^Ty\|_{\infty}$, the choice
of tuning parameter can be viewed as some multiple
of the expected threshold at which noise with variance $\sigma^2_E$ would
enter the LASSO path.

Unlike the LASSO, an analogous choice of tuning parameter $\lambda$ for square-root LASSO
does not depend on $\sigma_E$, 
\begin{equation}
\label{eq:lambda}
\lambda = \kappa \cdot \Ee \left(\frac{\|X^T\epsilon\|_{\infty}}{\|\epsilon\|_2} \right), \qquad \epsilon \sim N(0, I)
\end{equation}
for some unitless $\kappa$. Below, we typically use $\kappa \leq 1$.

Any spherically symmetric distribution yields the same choice of $\lambda$
which makes the above choice of tuning parameter independent of the noise level
$\sigma_E$. That $\lambda$ is independent of $\sigma_E$ also
follows from the convex program \eqref{eq:sqrt_lasso} since the first term and
the second term in the optimization objective are of the same order in
$\sigma_E$.

\subsection{Model selection by square-root LASSO}

Both the LASSO and square-root LASSO can be viewed as model selection 
procedures. In what follows we make the assumption that the columns of
$X$ are in general position to ensure uniqueness of solutions \cite{lassounique}.
We define the selected model of the square-root LASSO as
\begin{equation}
\label{eq:selected_model}
\hat{E}_{\lambda}(y) = \left\{j: \hat{\beta}_{j,\lambda}(y) \neq 0 \right\}
\end{equation}
and the selected signs
\begin{equation}
\label{eq:selected_signs}
\hat{z}_{E,\lambda}(y) = \text{sign}\left\{\hat{\beta}_{j,\lambda}(y): \hat{\beta}_{j,\lambda}(y) \neq 0 \right\}.
\end{equation}
To ease notation, we use the shorthands 
\begin{equation}
\label{eq:shorthand}
\hat{\beta}(y) = \hat{\beta}_{\lambda}(y), \quad
\hat{E} = \hat{E}_{\lambda}(y), \quad
\hat{z}_{\hat{E}} =\hat{z}_{E,\lambda}(y).
\end{equation}

In Section \ref{sec:sqrt_lasso} we investigate the KKT (Karush-Kuhn-Tucker) conditions for the program \eqref{eq:sqrt_lasso}. As in the LASSO case,
the KKT conditions provide the basic description for the {\em selection event}
on which selective inference is based.

\subsection{Selective inference}\label{sec:selective}

The selective inference framework described in \cite{optimal_selective} attempts to
control the {\em selective type I error rate} \eqref{eq:selective_type_1_error}. 
This is defined in terms of a pair of a model and an associated hypothesis $(M,H_0)$, and 
a critical function $\phi_{(M,H_0)}$ to test 
$H_0 \subset M$ vs. $H_a = M \setminus H_0$.
The selective type I error is
\begin{equation}
\Pp_{M,H_0}( \text{reject $H_0$}\ |\ \text{$(M, H_0)$ selected} ) 
= \Ee_{M,H_0}\left(\phi_{(M,H_0)}(y)\ | (M, H_0) \in \hat{\cal Q}(y)\right),
\label{eq:selective_type_1_error}
\end{equation}
where we use the notation $\hat{\cal Q}$ to denote the model selection procedure that 
depends on the data. 
The process $\hat{\cal Q}$ determines a map taking a distribution $\Pp \in M$ to
\begin{equation}
\label{eq:selective:dbn}
\Pp\left( \ \cdot \  \big \vert  (M, H_0) \in \hat{\cal Q}\right).
\end{equation}
We call such distributions {\em selective distributions}. Inference is carried out
under these distributions.

In this paper, we consider the models and hypotheses 
\begin{equation}
\label{eq:selected:unknown}
M = M_{u,E} = \left\{N(X_E \beta_E, \sigma^2_E I): \beta_E \in \real^E, \sigma^2_E \geq 0\right\}, \qquad
H_0 = \left\{\beta_E: \beta_{j,E}=0\right\},\quad j \in E.
\end{equation}

The setup of selective inference is such that we select the model $M_{u,E}$ based on  
a set of variables $\hat{E}$ selected by the data as in \eqref{eq:selected_model}. More specifically, 
\begin{equation}
\label{eq:questions_lasso}
\hat{\cal Q}(y) = \hat{\cal Q}_{u,\lambda}(y) = \left\{ (M_{u,\hat{E}}, \{\beta_{\hat{E}}:\beta_{j,\hat{E}} = 0 \}): j \in \hat{E} \right\}.
\end{equation}

One of the take-away messages from \cite{optimal_selective} is that
there is a concrete procedure to form valid tests that controls \eqref{eq:selective_type_1_error} 
when $M$ is an exponential family, and the null 
hypotheses 
can be expressed in terms of a one-parameter subfamily of the natural parameter space of the exponential family.
More specifically, we have the following lemma.

\begin{lemma}
\label{lem:suff}
For the regression model \eqref{eq:selected:unknown} with unknown $\sigma^2_E$, $(X_E^T y, \|y^2\|)$
are the sufficient statistics for the natural parameters $\big(\frac{\beta_E}{\sigma_E^2}, \frac{1}{\sigma_E^2}\big)$.
Furthermore, to test hypothesis $H_0:~ \frac{\beta_{j,E}}{\sigma_E^2} = \theta$, for any $j \in E$, we only need to
consider the law
\begin{equation}
\label{eq:law:coef}
{\cal L}_{(M_{u, E}, H_0)} \left(X_j^T y \mid \|y\|^2, X_{E \backslash j}^T y, (M_{u, E}, H_0) \in \hat{{\cal Q}}(y)\right).
\end{equation}

Similarly, the following law can be used for inference of $\sigma_E^2$, 
\begin{equation}
\label{eq:law:variance}
{\cal L}_{(M_{u, E}, \sigma_E^2)} \left(\|y\|^2 \mid X_{E}^T y, (M_{u, E}, \sigma_E^2) \in \hat{{\cal Q}}(y)\right).
\end{equation}
\end{lemma}

\begin{proof}
The proof is a direct application of Theorem 5 in \cite{optimal_selective}.
We have $(U(y), V(y)) = (X_j^T y, (\|y\|^2, X_{E \backslash j}^T y))$ for inference
of $\frac{\beta_{j,E}^2}{\sigma_E^2}$, and $(U(y), V(y)) = (\|y\|^2, X_E^T y)$ for
inference of $\sigma_E^2$. 
\end{proof}

Thus, by studying the distributions \eqref{eq:law:coef} and \eqref{eq:law:variance},
we will be able to perform inference after selection via the square-root LASSO.
To gain insight for the law in \eqref{eq:law:coef}, we first look at the
simple case where there is no selection. If there were no selection,
the above law is simply the law of the $T$-statistic
with degrees of freedom $n-|E|$,
$$
\dfrac{e_j^T X_E^{\dagger} y}{\hat{\sigma}_E \cdot \|e_j^T X_E^{\dagger}\|_2},
$$
where 
$$
\hat{\sigma}_E^2 = \frac{\|(I-P_E)y\|^2}{n-|E|}, \quad X_E^{\dagger} = (X_E^T X_E)^{-1} X_E^T,
\quad P_E = X_EX_E^{\dagger}.
$$

The selection event is equivalent to $\{\hat{E}(y) = E\}$ in the context
of this paper, where $\hat{E}$ is defined in \eqref{eq:selected_model}. 
We can explicitly describe the selection procedure
if we further condition on the signs $z_E$, that is instead of conditioning on
the event $\{\hat{E}(y) = E\}$, we condition on the event $\{(\hat{E}(y),\hat{z}_{\hat{E}}(y))=(E,z_E)\}$
in the laws \eqref{eq:law:coef} and \eqref{eq:law:variance}.
Procedures valid under such laws are also valid under those conditioned on $\{\hat{E}(y) = E\}$
since we can always marginalize over $\hat{z}_{\hat{E}}$. Therefore, for computational reasons
we always condition on $\{(\hat{E}(y),\hat{z}_{\hat{E}}(y))=(E,z_E)\}$.

We describe the distributions in \eqref{eq:law:coef}
in detail in Section \ref{sec:dbn}.  We will see that they are truncated $T$
distributions with the degrees of freedom $n - |E|$. Based on these laws, we
construct exact tests for the coefficients $\beta_E$. Given that the
appropriate laws in the case of $\sigma$ known are truncated Gaussian distributions, it is not
surprising that the appropriate distributions here are truncated $T$
distributions. To construct selective intervals, we suggest a natural Gaussian
approximation to the truncated $T$ distribution and investigate its performance
in a regression problem. Such approximation brought much convenience in computation. 

\subsection{Organization}

The take-away message of this paper is that selective inference with $\sigma^2_E$ unknown is possible in the
$n < p$ scenario using the square-root LASSO.
In Section \ref{sec:sqrt_lasso} we describe the square-root LASSO in more detail. 
In particular, we describe the selection events
\begin{equation}
\label{eq:selection_event}
\left\{y: (\hat{E}(y), \hat{z}_{\hat{E}}(y))=(E,z_E)\right\}, \qquad E \subset \{1, \dots, p\}, z_E \in \{-1,1\}^E.
\end{equation} 
Following this, in Section \ref{sec:dbn} we turn our attention to the main
inferential tools, the laws \eqref{eq:law:coef} and \eqref{eq:law:variance},
which allow us to perform inference for the coefficients $\beta_E$ and variance
$\sigma_E^2$ in the selected model. As an application of the p-values obtained
in Section \ref{sec:dbn}, we applied the BHq procedure
\cite{benjamini1995controlling} to these p-values and compare the FDR control
and power to another method designed to control FDR. In Section \ref{sec:application}, we also
compare the performance of our variance estimators $\hat{\sigma}^2$ with
other estimates of the variance.

\section{The Square Root LASSO}
\label{sec:sqrt_lasso}

We now use the \emph{Karush-Kuhn-Tucker} (KKT) conditions to describe the selection event.
Recall the convex program \eqref{eq:sqrt_lasso}\begin{equation}
\label{eq:sqrt_lasso:0}
\hat{\beta}(y) = \hat{\beta}_{\lambda}(y) = \argmin_{\beta \in \real^p} \|y-X\beta\|_2 + \lambda \cdot \|\beta\|_1
\end{equation}
as well as our shorthand for the selected variables and signs 
\eqref{eq:selected_model}, \eqref{eq:selected_signs}.

The KKT conditions characterize the solution as follows: $(\hbeta(y), \hz)$ is 
the solution and corresponding subgradient of \eqref{eq:sqrt_lasso:0} if and only if
\begin{gather}
    \frac{X^T(y-X\hbeta(y))}{\|y-X\hbeta(y)\|_2} = \lambda \cdot \hz \label{eq:kkt}
\\
\hz_j \in \BCAS \sign(\hat{\beta}_{j}(y)) & \text{if } j \in \hat{E}(y) \\ [-1, 1] & \text{if }j \not \in \hat{E}(y) . \ECAS
\end{gather}
We see that our choice of shorthand for $\hat{z}_{\hat{E}}$ 
corresponds to the $\hat{E}(y)$ coordinates of the subgradient
of the $\ell_1$ norm.

Our first observation, which we had not found in the literature
on square-root LASSO, is that square-root LASSO and LASSO have equivalent
solution paths. In other words, the square-root LASSO solution path is
a reparametrization of the LASSO solution path. Specifically,
\begin{lemma}
    \label{lem:soln_path}
For every $(E,z_E)$, on the event 
$\{(\hat{E}(y), \hat{z}_{\hat{E}}(y)) = (E,z_E)\}$ the solutions
of the LASSO and square-root LASSO are related as
\begin{equation}
\hat{\beta}(y) = \hat{\beta}_{\lambda}(y) = \tilde{\beta}_{\hat{\gamma}(y)}(y)
\end{equation}
where
\begin{equation}
\label{eq:gamma_hat}
\hat{\gamma}(y) = \lambda \hat{\sigma}_E(y) \cdot \left(\frac{n-|E|}{1- \lambda^2 \|(X_E^T)^{\dagger}z_E\|^2_2}\right)^{1/2}
\end{equation}
and
\begin{equation}
\label{eq:sigmaE}
\hat{\sigma}^2_E(y) = \frac{\|(I-X_EX_E^{\dagger})y\|^2_2}{n-|E|} = \frac{\|(I-P_E)y\|^2_2}{n-|E|}
\end{equation}
is the usual ordinary least squares estimate of $\sigma^2_E$ in the model $M_{u,E}$.
\end{lemma}

\begin{proof}
On the event in question, we can rewrite the KKT conditions using the fact $X \hbeta = X_E \hbeta_{E}$ as
\begin{gather}
X_E^T (y - X_E \hbeta_{E}(y)) = c_E(y) \cdot \lambda \cdot z_{E} \label{eq:kkt1} \\
X_{-E}^T (y - X_E \hbeta_{E}(y)) = c_E(y) \cdot \lambda \cdot z_{-E} \label{eq:kkt2} \\
\text{sign}(\hbeta_{E}(y)) = z_{E}, \quad \| \hz_{-E} \|_\infty < 1
\end{gather}
with
$$
\begin{aligned}
c_E(y) &= \|y-X_E\hat{\beta}_E(y)\|_2.
\end{aligned}
$$
Comparing \eqref{eq:kkt1} with the KKT conditions of LASSO in \cite{lee2013exact}, this indicates $\hat{\gamma}(y) = \lambda c_E(y)$. 
We deduce from \eqref{eq:kkt1} that the active part of the coefficients
\begin{equation}
\label{eq:lasso_soln}
\hat{\beta}_E(y) = (X_E^TX_E)^{-1}(X_E^Ty - \lambda \cdot c_E(y) \cdot z_E).
\end{equation}
Plugging the estimate $\hat{\beta}_E(y)$,
\begin{equation}
    \label{eq:norm}
y-X_E\hat{\beta}_E(y) = (I-P_E)y + \lambda \cdot c_E(y) \cdot (X_E^T)^{\dagger} z_E.
\end{equation}

Computing the squared Euclidean norm  of both sides yields
$$
\begin{aligned}
c_E^2(y) &= \frac{\|(I - P_E)y\|^2_2}{1 - \lambda^2 \|(X_E^T)^{\dagger} z_E\|^2_2} \\
&= \hat{\sigma}^2_E(y)\cdot  \frac{n-|E|}{1 - \lambda^2 \|(X_E^T)^{\dagger}z_E\|^2_2}.
\end{aligned}
$$ 

\end{proof}

\subsection{A first example}

Before we move on to the general case, it is helpful to look at the characterization of the selection event in the case of an orthogonal design matrix. 
When the design matrix $X \in \real^{n \times p}$ has orthogonal columns, $\hat{\beta}_E$ and $c_E$ can be simplified as
$$
\hat{\beta}_E = X_E^T y - \lambda c_E z_E, \quad c_E = \hat{\sigma}_E \cdot \displaystyle\sqrt{\frac{n-|E|}{1 - \lambda^2 |E|}}.
$$

The selection event 
$$
\left\{y: \text{sign}(\hbeta_{E}(y)) = z_{E} \right\},
$$
is decoupled into $|E|$ constraints. For each $i \in E$, 
\begin{equation}
\label{eq:trunc_t}
    \dfrac{z_i x_i^T y}{\hsigma_E(y)} \geq \lambda \displaystyle\sqrt{\frac{n-|E|}{1 - \lambda^2  |E|}}.
\end{equation}
The left-hand side of \eqref{eq:trunc_t} is closely related to the inference on $\beta_i$, and follows a $T$-distribution with $n-|E|$ degrees
of freedom. The constraint \eqref{eq:trunc_t} is a constraint on the usual $T$-statistic on the
selection event which implies one should use the truncated $T$ distribution to tests whether or not $\beta_i=0$.

\subsection{Characterization of the selection event}

We now describe the selection event for general design matrices, which will
be used for deriving the laws \eqref{eq:law:coef} and \eqref{eq:law:variance}.
Using Equation \eqref{eq:lasso_soln} in Lemma \ref{lem:soln_path}, we see that the event
\begin{equation}
\label{eq:active_constraints0}
\left\{y: \text{sign}(\hbeta_{E}(y)) = z_{E} \right\}
\end{equation}
is equal to the event
\begin{equation}
\label{eq:active_constraints}
\left\{y: \hat{\sigma}_E(y) \cdot \alpha_{i,E} - z_{i,E} \cdot U_{E,i}(y) \leq 0, i \in E \right\}
\end{equation}
where
$$
U_{E,i}(y) = \frac{ e_i^T X_E^{\dagger} y }{\|e_i^TX_E^{\dagger}\|_2}
$$
and
$$
\alpha_{i,E} = \lambda \cdot z_{i,E} \cdot \|e_i^TX_E^{\dagger}\|_2 \cdot (1- \lambda^2 \|(X_E^T)^{\dagger}z_E\|^2_2)^{-1/2} e_i^T(X_E^TX_E)^{-1} z_E.
$$
While the expression is a little involved, it is explicit and easily
computable given $(X_E^TX_E)^{-1}$.

Let us now consider the inactive inequalities. The event
\begin{equation}
\label{eq:inactive_constraints0}
\left\{y: \|\hat{z}_{-E}\|_{\infty} < 1 \right\}
\end{equation}
is equal to the event
\begin{equation}
\label{eq:inactive_constraints}
\left\{y:\left|\frac{X_i^T(I-P_E)y}{\lambda \cdot c_E(y)} + X_i^T(X_E^T)^{\dagger}z_E \right| < 1, \qquad i \in -E \right\}.
\end{equation}

For each $i \in -E$ these are equivalent the intersection of the inequalities
\begin{equation}
\begin{aligned}
\label{eq:inactive2}
\left(\frac{1 - \lambda^2 \|(X_E^T)^{\dagger}z_E\|^2_2}{\lambda^2} \right)^{1/2}X_i^TU_{-E}(y) &< 1 - X_i^T (X_E^T)^{\dagger}z_E \\
\left(\frac{1 - \lambda^2 \|(X_E^T)^{\dagger}z_E\|^2_2}{\lambda^2} \right)^{1/2}X_i^TU_{-E}(y) &> -1 - X_i^T (X_E^T)^{\dagger}z_E. \\
\end{aligned}
\end{equation}
where
\begin{equation}
\begin{aligned}
\label{eq:ancillary}
U_{-E}(y) = \frac{(I-P_E)y}{\|(I-P_E)y\|_2}.
\end{aligned}
\end{equation}

To summarize, the selection event $\left\{y: (\hat{E}(y), \hat{z}_{\hat{E}}(y))=(E,z_E)\right\}$ is
equivalent to \eqref{eq:active_constraints} and \eqref{eq:inactive2}.

\section{The conditional law}
\label{sec:dbn}

From the characterization of the selection event in the above section
we can now derive the laws \eqref{eq:law:coef} and \eqref{eq:law:variance}. 
First, the following lemma provides some simplification.

\begin{lemma}
\label{lem:equivalent}
Conditioning on $(\hat{E}(y), \hat{z}_{\hat{E}}(y)) = (E, z_E)$, the law for
inference of $\big(\frac{\beta_E}{\sigma^2_E}, \frac{1}{\sigma^2_E}\big)$
is equivalent to
\begin{equation}
\label{eq:conditional:law:simple}
\mathbb{Q}_{E, z_E} = {\cal L}\left[U_E(y), \|(I-P_E)y\|^2 \mid \hat{\sigma}_E \cdot \alpha_{i,E} - z_{i,E} \cdot U_{E,i} \leq 0, i \in E \right].
\end{equation}
Moreover, the statistic $U_{-E}$ is ancillary for both $\Pp_E \in M_{u,E}$ and $\Q_{E, z_E}$.

Therefore, the laws \eqref{eq:law:coef} and \eqref{eq:law:variance} can be simplified to
$$
{\cal L}\left[U_{E,j}(y) \mid \|y\|^2, U_{E, E\backslash j}(y), \hat{\sigma}_E \cdot \alpha_{i,E} - z_{i,E} \cdot U_{E,i} \leq 0, i \in E\right],
$$ 
and 
$$
{\cal L}\left[\hat{\sigma}_E^2(y) \mid U_{E}(y), \hat{\sigma}_E \cdot \alpha_{i,E} - z_{i,E} \cdot U_{E,i} \leq 0, i \in E\right].
$$ 
respectively.
\end{lemma}

\begin{proof}
Per Lemma \ref{lem:suff}, 
$(X_E^T y, \|y\|^2)$ are sufficient statistics for $\big(\frac{\beta_E}{\sigma_E^2}, \frac{1}{\sigma_E^2}\big)$.
Moreover, note $U_E(y)$ is linear transformation of $X_E^T y$ and
$$
\|y\|^2 = \|P_E y\|^2 + \|(I-P_E)y\|^2, \qquad P_E = X_E (X_E^T X_E)^{-1} X_E^T,
$$
thus it suffices to consider the law of $(U_E(y), \|(I-P_E)y\|^2)$ conditioning on 
$\{(\hat{E}(y), \hat{z}_{\hat{E}}(y)) = (E, z_E)\}$.

Moreover, the selection event can be described in two sets of constraints,
the ones involving $U_E(y)$ in \eqref{eq:active_constraints} and the ones involving
$U_{-E}(y)$ in \eqref{eq:inactive2}. Notice that $U_{-E}(y)$ is independent of $(U_E(y), \hat{\sigma}_E^2(y))$,
thus we can drop it in the conditional distribution and get \eqref{eq:conditional:law:simple}. 

We note that $U_{-E}(y)$ is ancillary for $\Pp_E$. From the above, 
the density of $\Q_{E, z_E}$ is that of $\Pp_E$ times an indicator function
that does not involve $U_{-E}(y)$. Thus $U_{-E}(y)$ is ancillary for $\Q_{E, z_E}$ as well.

The simplification of the joint laws leads to that of the marginal laws and the second
statement holds.
\end{proof}

\subsection{Inference under quasi-affine constraints}

The general form of the law $\Q_{E, z_E}$ in \eqref{eq:conditional:law:simple} is
that of a multivariate
Gaussian and an independent $\chi^2$ of degrees of freedom $n-|E|$ 
and satisfying some constraints. These constraints are affine in the Gaussian
fixing the $\chi^2$, but not affine in the data. To study the distributions
under such constraints, we propose the following framework for these
{\em quasi-affine constraints}. 
 
Specifically, we want to study the distribution of $y \sim N(\mu, \sigma^2 I)$ 
subject to quasi-affine constraints
\begin{equation}
\label{eq:quasi:affine}
Cy \leq \hat{\sigma}_P(y) \cdot b
\end{equation}
with 
$$
\hat{\sigma}^2_P(y) = \frac{\|(I-P)y\|^2_2}{\text{Tr}(I-P)}
$$
where $P$ is a projection matrix, $C \in \real^{d \times p}$, $b \in \real^d$ and
\begin{equation}
\label{eq:subspace}
CP = P, \quad P\mu = \mu.
\end{equation}
Assumptions \eqref{eq:subspace} are made to simplify notation. Inference for quasi-affine
constraints without these assumptions can be deduced similarly. In the example of inference after
square-root Lasso, assumptions \eqref{eq:subspace} are satisfied when we select the correct model $E$.
We denote the above laws as $\mathbb{M}_{C,b,P}$, that is  
$$
\Pp(y \in A \vert Cy \leq \hat{\sigma}_P(y) \cdot b) \overset{\Delta}{=} \M_{(C,b,P)}(A), \qquad y \sim N(\mu, \sigma^2 I).
$$

Our goal is exact inference for $\eta^T\mu$, for the directional vector $\eta$ satisfying $P\eta=\eta$.
Without loss of generality we assume $\|\eta\|^2_2=1$.
To test the null hypothesis $H_0:\eta^T\mu=\theta$, we parametrize the data into
$\eta^T y$, the projection onto the direction $\eta$, and the orthogonal direction
$(P - \eta \eta^T)y$. From the assumptions \eqref{eq:subspace}, with some algebra,
we can see that $\eta^T y$ is a sufficient statistic for $\eta^T \mu$, with
$((P - \eta \eta^T)y, \|y\|^2)$ being a sufficient statistic for the nuisance parameters.
In the following, we will prove the law
\begin{equation}
\eta^Ty - \theta \ \big \vert \ (P - \eta\eta^T)y, \|y-\theta \eta\|^2_2 \qquad y \sim \M_{(C,b,P)}.
\end{equation}
is a truncated $T$
with degrees of freedom $\text{Tr}(I-P)$ and an explicitly computable
truncation set.

For some set $\Omega \subset \real$ let $T_{\nu |\Omega}$ denote the distribution function of the law of
$T_{\nu} | T_{\nu} \in \Omega:$
\begin{equation}
T_{\nu |\Omega}(t) = \Pp(T_{\nu} \leq t \vert T_{\nu} \in \Omega).
\end{equation}

\begin{theorem}[Truncated $t$]
\label{thm:truncated:t}
Suppose that $y \sim \M_{(C,b,P)}$.
The law 
\begin{equation}
\eta^Ty - \theta \ \big \vert \ (P - \eta\eta^T)(y-\theta\eta), \|y-\theta\eta\|^2_2 \overset{D}{=} T_{\text{Tr}(I-P) |\Omega}
\end{equation}
where
\begin{equation}
\Omega = \Omega(C, b, P, \|(I-P)y\|^2_2 + (\eta^Ty-\theta)^2, (P-\eta\eta^T)y, \theta).
\end{equation}
The precise form of $\Omega$ is given in \eqref{eq:slice} below.
\end{theorem}

\begin{proof}
We use the short hand for the sufficient statistics 
$$
(U_{\theta}, V, W_{\theta})(y) 
= (\eta^Ty-\theta, (P - \eta\eta^T)y, \|(I-(P-\eta^T\eta))(y-\theta \eta)\|^2_2).
$$
Since
$$
W_{\theta}(y) = \|y - \theta \eta\|^2 - \|(P- \eta\eta^T)y\|^2,
$$
conditioning on $(V, \|y - \theta \eta\|^2)$ is equivalent to conditioning on $(V, W_{\theta})$.

Our main strategy is to construct a test statistic independent of $(V, W_{\theta})$, this can be
easily done through the usual T-statistic,
$$
\tau_{\theta}(y) = \frac{\eta^Ty-\theta}{\hat{\sigma}_P(y)}.
$$
Let $d=\text{Tr}(I-P)$, so
$$
W_{\theta}(y) = \|(I-P)y\|^2_2 + (\eta^Ty-\theta)^2
= d \cdot \sigma_P(y)^2 + (\eta^Ty-\theta)^2.
$$
Note that $\tau_{\theta}$ is independent of $W_{\theta}$ and $V$.

We next rewrite the quasi-affine inequalities \eqref{eq:quasi:affine} as
$$
\begin{aligned}
U_{\theta}(y) \nu + \xi &\leq \sigma_P(y) b
\end{aligned}
$$
with 
$$
\begin{aligned}
\nu &= C\eta \\
\xi &= \xi(V(y)) =  C(\theta \eta + V(y)).
\end{aligned}
$$
Multiplying both sides by $\frac{W_{\theta}(y)^{1/2}}{\sigma_P(y)}$, we have
$$
\begin{aligned}
\tau_{\theta}(y)W_{\theta}(y)^{1/2}  \nu  + \xi(V(y)) \cdot \sqrt{d + \tau_{\theta}^2(y)} &\leq W_{\theta}(y)^{1/2} b.
\end{aligned}
$$
This is the constraint on the $T$-distribution.

Because of the independence between $\tau_{\theta}(y)$ and $(V(y), W_{\theta}(y))$, its distribution is simply
a truncated $T$-distribution, $T_{d | \Omega}$, where
\begin{equation}
\label{eq:slice}
\Omega(C,b,P,w,v,\theta) = \bigcap_{1 \leq i \leq \text{nrow}(A)} \left\{t \in \real: t \sqrt{w} \cdot \nu_i + \xi_i(v) \cdot \sqrt{d + t^2} \leq \sqrt{w} \cdot b \right\}
\end{equation}
and $d, \xi(v)$ are as above.
Each individual inequality can be solved explicitly, with each one yielding at most 2 intervals. In practice, we have observed the intersection of the above is
not too complex.
\end{proof}

\begin{remark}
Following Lemma \ref{lem:equivalent}, we conduct selective inference with the quasi-affine constraints of square-root Lasso, where
$$
C = - \diag(z_E) \diag((X_E^T X_E)^{-1}) X_E^{\dagger}, \quad P = P_E, \quad b = -\alpha.
$$
To test hypothesis $H_0: \beta_{j,E} = 0$, we take $\eta = e_j X_E^{\dagger}$.
\end{remark}

\subsection{Inference for $\sigma$ and debiasing under $\M_{(C,b,P)}$}
\label{sec:sigma}

The law $\M_{(C,b,P)}$ is parametric, and in the context of model
selection we have observed $y$ inside the set
$$
Cy \leq \hat{\sigma}_P(y) \cdot b.
$$
The usual OLS estimates
$
X_E^{\dagger}y$ are biased under $\M_{(C,b,P)}=\M_{(C,b,P)}(\mu,\sigma^2)$. 
In this parametric
setting, there is a natural procedure to attempt to debias
these estimators.

If we fix the sufficient statistics to be
$T(y)=(Py, \|y\|^2_2)$, then the natural parameters 
of the laws in $\M_{(C,b,P)}$ are $(\mu/\sigma^2, -(2\sigma^2)^{-1})$.
Solving the score equations
\begin{equation}
\label{eq:MLE}
\int_{\real^n} T(z) \; \M_{(C,b,P);(\mu,\sigma^2)}(dz) - T(y) = 0
\end{equation}
for $(\hat{\mu}(y), \hat{\sigma}^2(y))$
corresponds to selective maximum likelihood estimation under $\M_{(C,b,P)}$.
In the orthogonal design and known variance setting, this problem
was considered by \cite{reid2013study}. 
In our current setting,
this requires sampling from the constraint set, which is generally non-convex.

Instead, we consider estimation of each parameter separately based on
a form of pseudo-likelihood.
Unfortunately, in the
unknown variance setting,
this approach yields estimates either for coordinates of
$\mu/\sigma^2$ or $\sigma^2$ rather
than coordinates of $\mu$ itself. We propose estimating 
$\sigma^2$ using pseudo-likelihood and plugging in this value to a
quantity analogous to $\M_{(C,b,P)}$ but with known variance, i.e. the law of
$y \sim N(\mu, \sigma^2 I), P\mu=\mu$ with $\sigma^2$ known subject to an affine 
constraint. This approximation is discussed in Section \ref{sec:approx} below.

The pseudo-likelihood is based on the law of one
sufficient statistic conditional on the other
sufficient statistics. Therefore, to estimate
$\sigma^2$ we consider the likelihood based on the law
\begin{equation}
\|(I-P)y\|^2_2 \big \vert Py, \qquad y \sim \M_{(C,b,P);(\mu,\sigma^2)}.
\end{equation}
This law depends only on $\sigma^2$ and can be used for exact inference 
about $\sigma^2$, though for the parameter $\sigma^2$ an
estimate is perhaps more useful than selective tests or 
selective confidence intervals.

Direct inspection of the inequalities yield that this law is equivalent to
$\sigma^2 \cdot \chi^2_{\text{Tr}(I-P)}$ 
truncated to the interval $[L(Py), U(Py)]$ where
\begin{equation}
\begin{aligned}
L(Py) &= \max_{i:b_i \geq 0} \frac{(Cy)_i}{b_i} \\
U(Py) &= \min_{i:b_i \leq 0} \frac{(Cy)_i}{b_i}.
\end{aligned}
\end{equation}

For $\Omega \subset \real$, let $G_{\nu,\sigma^2,\Omega}$ denote the law $\sigma^2 \cdot \chi^2_{\nu}$ truncated to $\Omega$
$$
G_{\nu,\sigma^2,\Omega}(t) = \Pp \left(\chi^2_{\nu} \leq t | \chi^2_{\nu} \in \Omega / \sigma^2\right).
$$
The pseudo-likelihood estimate $\hat{\sigma}_{PL}(y)$ for $\sigma^2$ is the 
root of
\begin{equation}
\sigma \mapsto H_{\text{Tr}(I-P)}(L(Py), U(Py), \sigma^2) -  \hat{\sigma}^2_P(y)
\end{equation}
where
$$
H_{\nu}(L, U, \sigma^2) = \frac{1}{\nu} \int_{[0,\infty)} t \; G_{\nu,\sigma^2,[L,U]}(dt).
$$
This is easily solved by sampling from $\sigma^2 \chi^2_{\text{Tr}(I-P)}$ truncated to $[L(Py),U(Py)]$. 

\begin{figure}
\centering
\begin{subfigure}[t]{0.45\textwidth}
  \includegraphics[width=\textwidth]{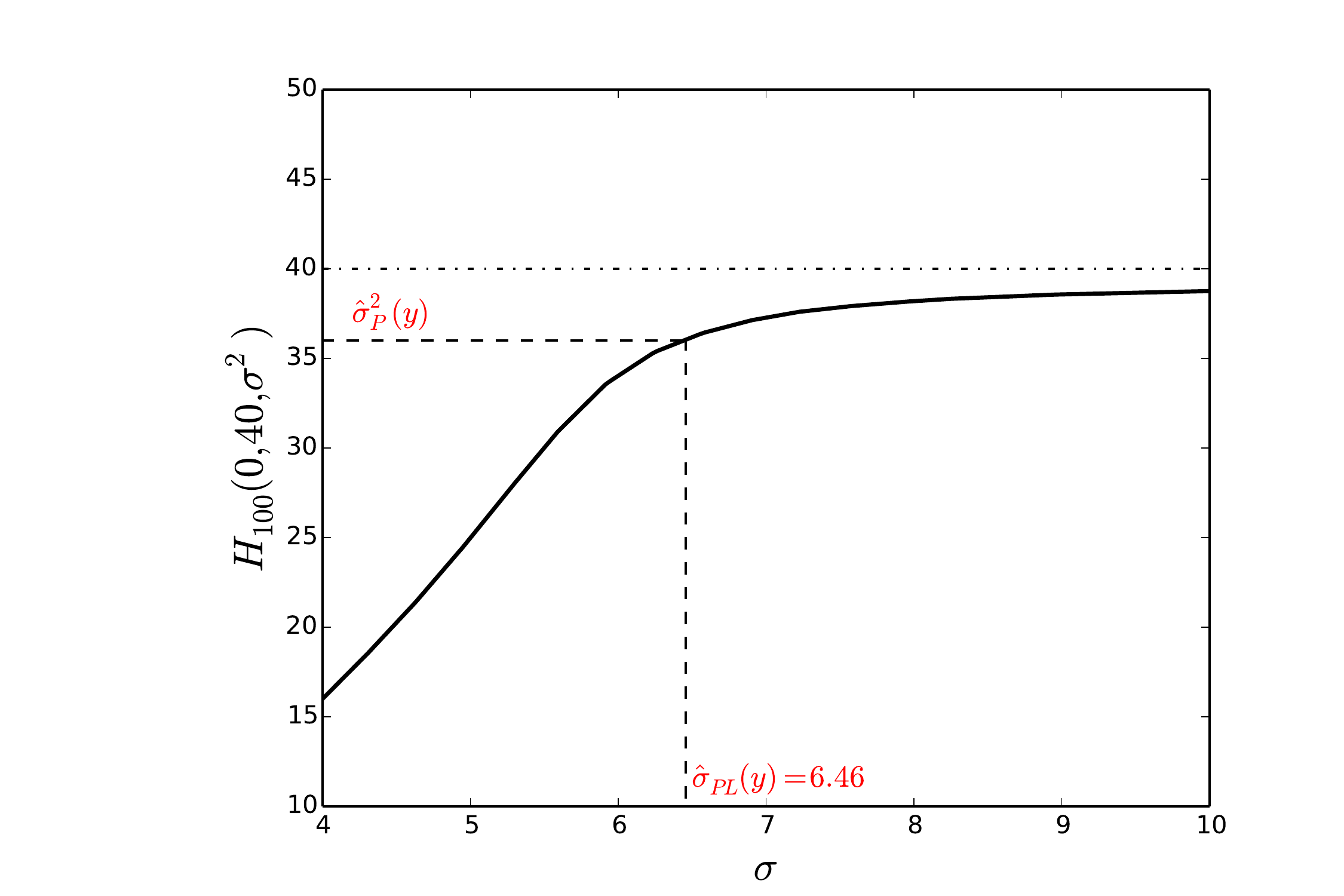}
  \caption{Pseudo-likelihood estimator}
  \label{fig:estimating_sigma}
\end{subfigure}
~ 
\begin{subfigure}[t]{0.45\textwidth}
  \includegraphics[width=\textwidth]{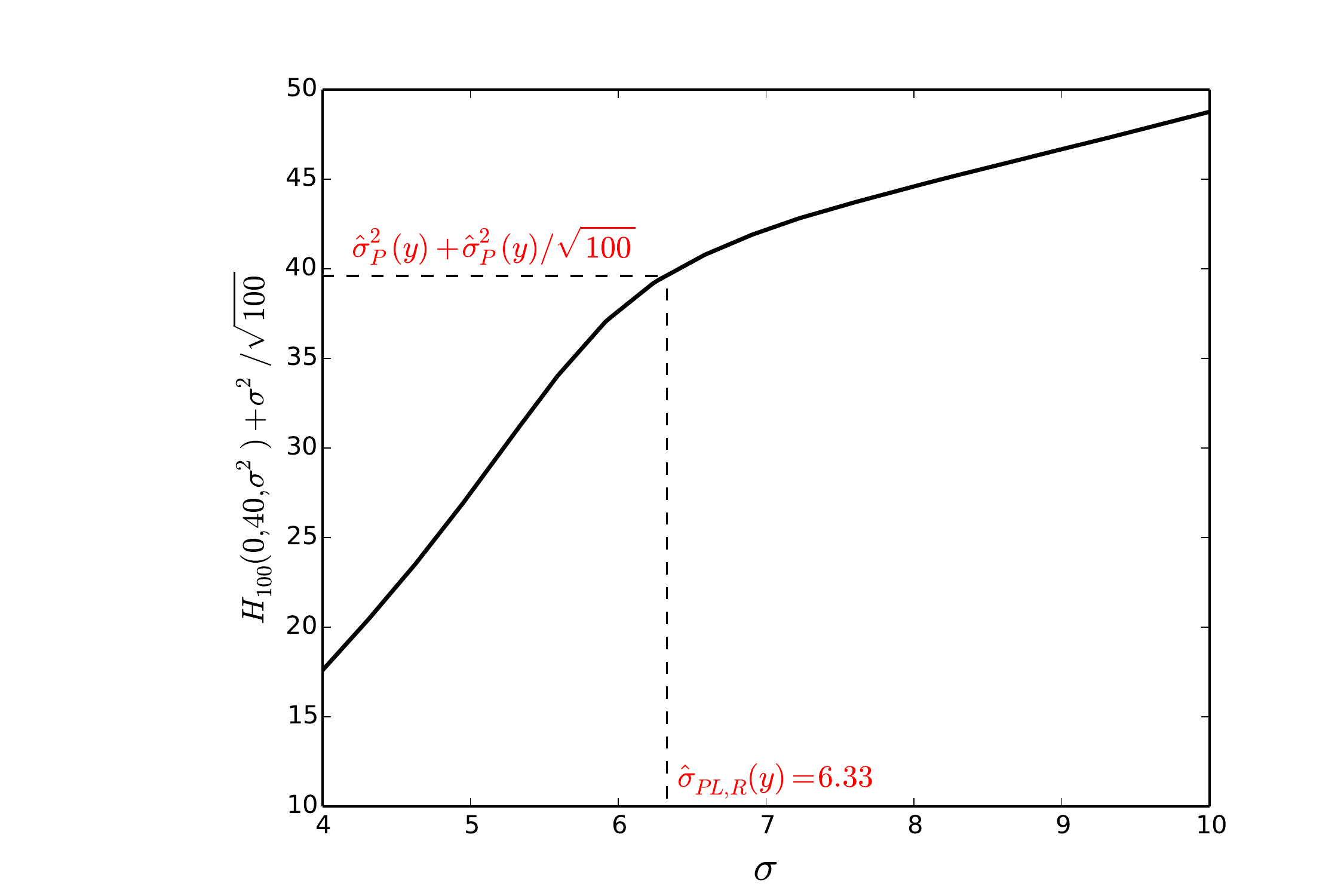}
  \caption{Regularized pseudo-likelihood estimator}
  \label{fig:lassopvalb}
\end{subfigure}
\caption{Estimation of $\sigma^2$ based on the pseudo-likelihood
for $\sigma^2$ under $\M_{(C,b,P)}$ for an interval with $L=0$.
The observed value is $\hat{\sigma}^2_P(y)=36$ on 100 degrees of freedom
truncated to $[0,40]$.
}
\end{figure}

This procedure is illustrated in 
Figure \ref{fig:estimating_sigma}. Note that
for observed values of $\hat{\sigma}^2_P(y)$ near the truncation boundary
the estimate varies quickly with $\hat{\sigma}^2_P(y)$ due to the plateau
at the upper limit.
We remedy this in two simple steps. First, we use a regularized estimate of
$\sigma$ under this pseudo-likelihood. 
 Next, we apply a simple bias correction
to this regularized estimate so that when $[L,U] = [0,\infty)$ we recover
the usual OLS estimator. Specifically, for
some $\theta$ we obtain a new estimator as the root of
$$
\sigma \mapsto H_{\nu}(L, U, \sigma^2) + \theta \cdot \sigma^2 - (1 + \theta) \hat{\sigma}^2_P(y)
$$
We call this regularized pseudo-likelihood estimate $\hat{\sigma}^2_{PL,R}(y)$.
In practice, we have set $\theta=\nu^{-1/2}$ so that this regularization
becomes negligible as the degrees of freedom grows.
The regularized estimate
can be thought of as the MAP from an improper prior on the natural parameter
for $\sigma^2$. In this case, if $\delta=1/(2\sigma^2)$ is the natural parameter
the prior has density proportional to $\delta^{\nu \cdot \theta}$. As
$H_{\nu}(0,\infty, \sigma^2)=\sigma^2$, it is clear that
in the untruncated case we recover the usual OLS estimator $\hat{\sigma}_P(y)$.

\section{Applications}
\label{sec:application}

In this section, we discuss several applications of the inferential tools introduced above.
In Section \ref{sec:compare_estimator}, we compare the performance of the variance estimator
introduced in Section \ref{sec:sigma} with some existing methods. In Section \ref{sec:approx},
we introduce a Gaussian approximation to the truncated $T$-distribution, which is more
computation-friendly. The approximation is
validated through simulations where the coverages of confidence intervals are close to the nominal
levels. Finally, in Section \ref{sec:fdr}, we apply a BHq procedure to the p-values acquired through
the inference above. Although we do not seek to establish any theoretical results, simulations show that a simple
BHq procedure applied to the p-values controls the false discovery rate at the desired level and has
comparable power with existing methods, even in the high-dimensional setting. 

\subsection{Comparison of estimators}
\label{sec:compare_estimator}

As the selection event for the square-root LASSO yields
a law of the form $\M_{(C,b,P)}$, we can study
the accuracy of the estimator by comparing it to other
estimators for $\sigma$ in the LASSO literature.
We compare our estimator, $\hat{\sigma}_{P,L,R}$, to the following:
\begin{itemize}
    \item OLS estimator in the selected model, where $\hat{\sigma} = \frac{\|(I-P_E)y\|}{\sqrt{n-|E|}}$, 
        $E$ is the active set.
    \item Scaled LASSO in \cite{sun2012scaled}.
    \item Minimum cross-validation estimator based upon the residual sum of squares of Lasso solution with $\lambda$ selected by cross validation \cite{reid2013study}. 
\end{itemize}

To illustrate the advantage of our method, we consider the high-dimensional setting.
The design matrices were $1000 \times 2000$ generated from an equicorrelated Gaussian with correlation 0.3, columns normalized to have length 1.
The sparsity was set to $40$ non-zero coefficients each with a signal-to-noise ratio $7$ but with a random sign. The noise level $\sigma=3$ is
considered unknown. The parameter $\kappa$ 
in \eqref{eq:lambda} was set to $0.8$. With these settings, the square-root LASSO ``screened'', or discovered a superset of the
20 non-zero coefficients with a success rate of approximately 30\%. Since we do not screen most of the time, the model
may be sometimes misspecified. But the variance estimator $\hat{\sigma}_{PL}$ is consistent with the model specific
variance $\bar{\sigma}_E$. 

Specifically, the performance of the estimators was evaluated by considering
the ratio 
$\widehat{\sigma}^2(y)/\bar{\sigma}^2_E$ where 
$\bar{\sigma}^2_E$ is the usual estimate of $\sigma^2$ using the selected variables $E$, evaluated on
an independent copy of data drawn from the same distribution. This is the variance one would expect
to see in long run sampling if fitting an OLS model with variables $E$ under the true data generating
distribution.

\begin{figure}
\centering
\begin{subfigure}[ht]{0.45\textwidth}
  \includegraphics[width=\textwidth]{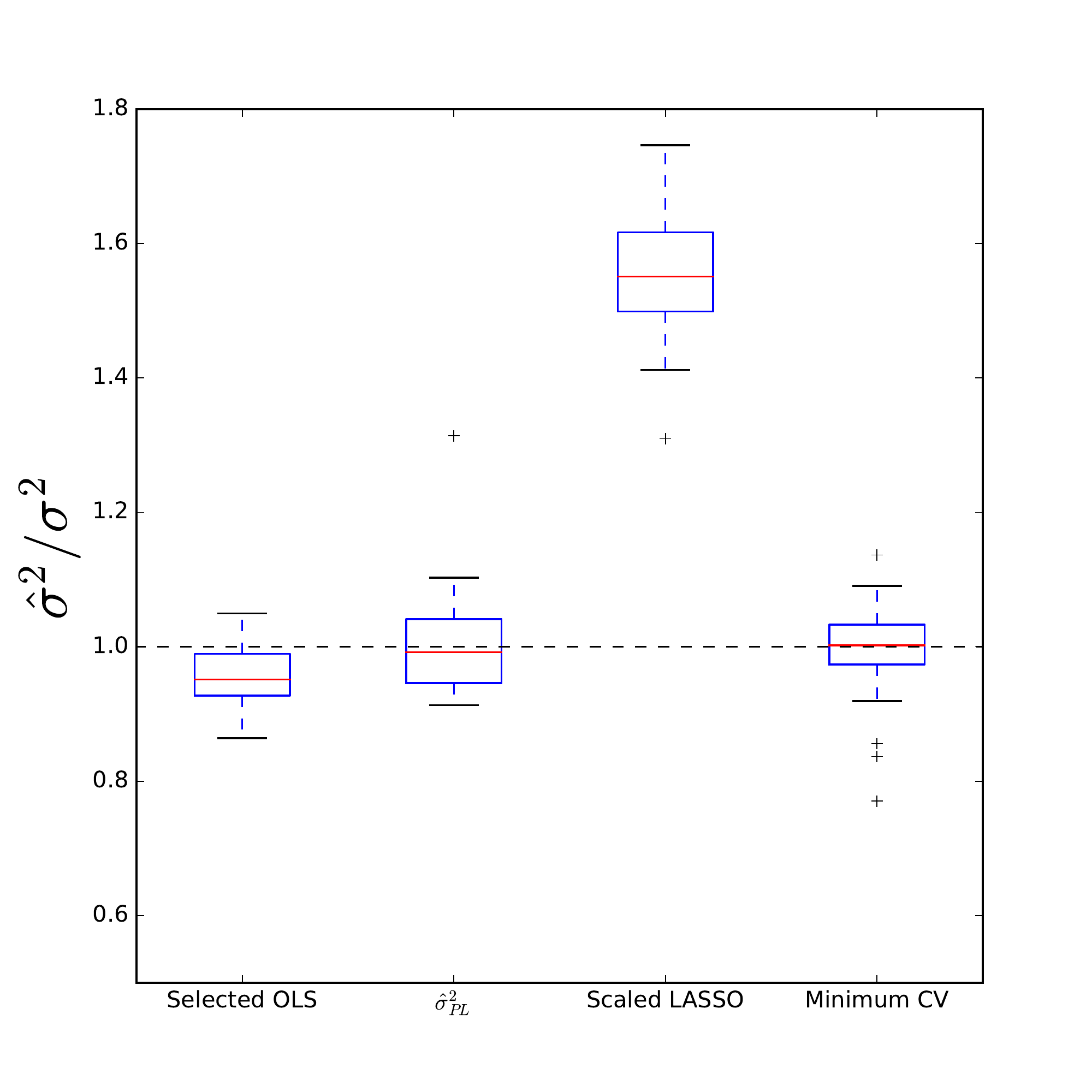}
  \caption{Screening}
  \label{fig:screened}
\end{subfigure}
~ 
\begin{subfigure}[ht]{0.45\textwidth}
  \includegraphics[width=\textwidth]{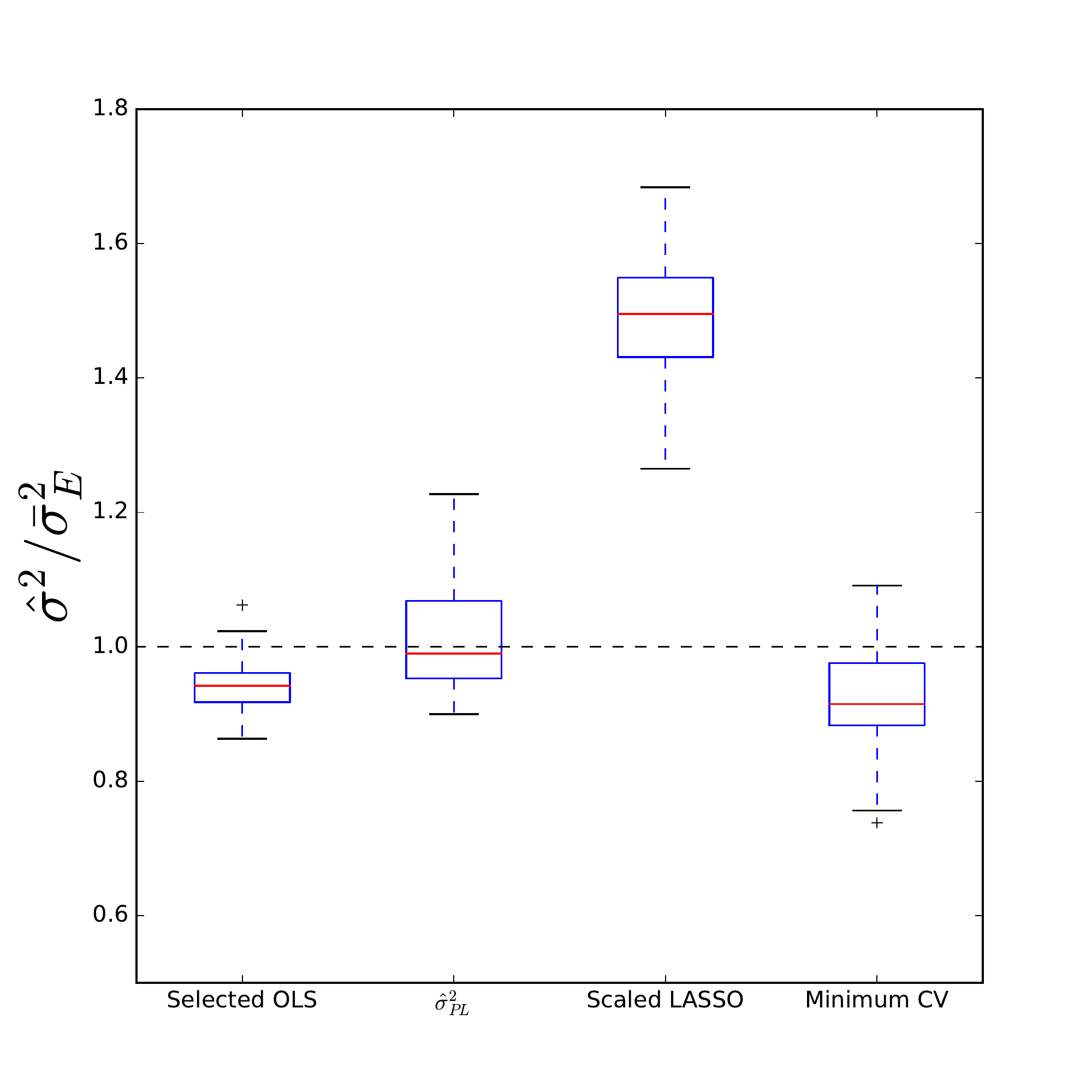}
  \caption{Non-screening}
  \label{fig:unscreened}
\end{subfigure}
\caption{Comparisons of different estimators for $\sigma_E$.
}
\end{figure}

We see from Figure \ref{fig:screened} that our estimator is approximately unbiased when we correctly recovered all variables,
outperforming the OLS estimator and the scaled LASSO estimator. Its performance is comparable with that of Minimum CV, but with
fewer outliers. In the case of partial recovery (non-screening), we see
that the estimator is quite close to the estimator $\bar{\sigma}^2_E$, where part of the ``noise'' in
our selected model comes from the bias in the estimator. Our estimator beats all the other estimators in this case.
Particularly, the Selected OLS estimator consistently underestimates the variance, which might lead to
inflated test scores and more false discoveries. On the other hand, Minimum CV seeks to estimate $\sigma^2$ 
instead of $\bar{\sigma}^2_E$, which results in large downward bias.

In fact, when $(y_i,X_i)$ are independent draws from a fixed Gaussian distribution, then for any $E$, the model $M_{u,E}$ is correctly specified in the sense that the law of $y|X_E$ belongs to the family $M_{u,E}$. In this setting the quantity $\bar{\sigma}^2_E$ is an asymptotically correct estimator $\sigma^2_E = \text{Var}(y_i | x_{i,E})$. Hence,
we see that the pseudo-likelihood estimator may be considered a reasonable estimator when the square-root LASSO does not actually screen.

\subsection{A Gaussian approximation to $\M_{(C,b,P)}$}
\label{sec:approx}

In this section, we introduce a Gaussian approximation to the truncated
$T$-distribution which we use in computation. 

The selective distribution $\mathbb{Q}_{E,z_E}$ derived from some $\P_E \in
M_{u,E}$ is used for inference about the parameters $\beta_{E}$.  Let
$\eta$ be the normalized linear functional for testing $H_0:\beta_{j,E}=\theta$, then for any value of
$\theta$, this distribution is restricted to the sphere of radius $\|y - \theta X_j\|_2$
intersect the affine space $\{z: X_{E \setminus j}^Tz=X_{E \setminus j}^Ty\}$. Call this set
$S(\|P_{E \setminus j}(y - \theta \eta)\|_2, X_{E \setminus j}^Ty)$
 The restriction of $\mathbb{Q}_{E,z_E}$ to $S(\|P_{E \setminus j}(y - \theta \eta)\|_2, X_{E \setminus j}^Ty)$ is 
the law $\P_E$ restricted to $S(\|P_{E \setminus j}(y - \theta \eta), X_{E \setminus j}^Ty)$ intersected with the selection event.
For $|E|$ not large relative to $n$ by the classical 
Poincar\'e's limit \cite{diaconis_freedman}, the law of
$\eta^Ty - \theta$ under $\P_E$ restricted to $S(\|P_{E \setminus j}(y - \theta \eta)\|_2, X_{E \setminus j}^Ty)$
that is close to a Gaussian with variance $\|P_{E \setminus j}(y - \theta \eta)\|_2^2 / (n - |E|+1)$. 
Thus, we might approximate its distribution under $\Q_{E,z_E}$ by a truncated Gaussian. 
Furthermore, we can approximate the variance by $\hat{\sigma}_P(y)^2$. We summarize the Gaussian approximation
in the following remark.

\begin{remark}[Approximate distribution]
\label{rem:approximation}
Suppose we are interested in testing the hypothesis $H_0:\eta^T\mu=\theta$ in the family $\M_{(C,b,P)}$
for some $\eta \in \text{row}(C)$. We propose using the distribution
$$
{\cal L}(\eta^Tz - \theta | (P - \eta\eta^T)z, Cz \leq \widehat{\sigma}_P(y) b), \qquad z \sim N(\mu, \hat{\sigma}^2(y) I).
$$
We condition on $(P - \eta\eta^T)z$ as we have assumed $P\mu=\mu$ in defining $\M_{(C,b,P)}$ and this is a sufficient
statistic for the unknown parameter $(P-\eta\eta^T)\mu$.
\end{remark}

To validate the approximation, we use a similar simulation scenario to the one in Section 7 of \cite{optimal_selective}.
We generate rows of the design matrix $\bX_{150 \times 200}$ from an equicorrelated multivariate Gaussian distribution with pairwise correlation $\rho=0.3$ between the variables. The columns are normalized to
have length 1. The sparsity level is 10, with each non-zero coefficient having value 6. Results are shown in Table~\ref{tab:coverage}.

\begin{table}[ht]
\centering
\begin{tabular}{rr}
\toprule
 Level &  Coverage \\
\midrule
 0.850 &     0.860 \\
 0.900 &     0.905 \\
 0.950 &     0.947 \\
 0.970 &     0.968 \\
\bottomrule
\end{tabular}

\caption{Coverage of confidence intervals using Gaussian approximation based on forming 10000 intervals.}
\label{tab:coverage}
\end{table}

\subsection{Regression diagnostics}

Recall the scaled residual vector $U_{-E}(y)$ in \eqref{eq:ancillary} is ancillary under the laws $\Pp_E$  and $\tilde{\Q}_{E,z_E}$. $U_{-E}(y)$ follows a uniform distribution on the $n$-dimensional unit sphere intersecting the subspace determined by $I-P_E$, truncated by the observed constraints \eqref{eq:inactive2}. 
As $U_{-E}(y)$ is ancillary, we can sample from its distribution to carry out any regression diagnostics or goodness of fit tests.
For a specific example, we might consider the observed maximum of the residuals $\|\hat{U}_{-E}(y)\|_{\infty}$.

Another natural regression diagnostic might be to test whether individual or groups of variables not selected improve the fit. Specifically, suppose
$G$ is a subset of variables disjoint from $E$. Then, the usual $F$ statistic for including these variables in the model 
is measurable with respect to $U_{-E}$:
$$
\begin{aligned}
F_{G|G \cup E}(y) &= \frac{\|(P_{G \cup E} - P_E)y\|^2_2 / |G|}{\|(I - P_{G \cup E})y\|^2_2 / (n-|G \cup E|)} \\
&= \frac{\|(P_{G \cup E} - P_E)U_{-E}(y)\|^2_2 / |G|}{\|(I - P_{G \cup E})U_{-E}(y)\|^2_2 / (n-|G \cup E|)}.
\end{aligned}
$$
Therefore, a selectively valid test of 
$$
H_0: \beta_{G|G \cup E} = 0
$$
can be constructed by sampling $U_{-E}$ under its null distribution and comparing the observed $F$ statistic to this reference distribution. Details of these diagnostics are a potential area of further work.

\subsection{Applications to FDR control}
\label{sec:fdr}

Based on the truncated-$t$ distribution derived in Theorem \ref{thm:truncated:t}, 
it is easy to construct tests that control the selective Type I error \eqref{eq:selective_type_1_error}.
In this section, we attempt to use our $p$-values for multiple hypothesis testing purposes.
We apply the BHq procedure \citep{benjamini1995controlling} to the p-values and compare with
the procedure proposed by \cite{barber2016knockoff}. To ensure the fairness of the comparison,
we generate the data according to the data generation mechanism in Section 5 of \cite{barber2016knockoff},
where $n=2000,~p=2500,~k=30$, and $X$ is generated as random Gaussian design with correlation $\rho$.
In the simulations, we vary $\kappa$ in \eqref{eq:lambda} and also include the choice of $\lambda$
according to the $1$ standard error rule for reference.
The comparison is in Figure \ref{fig:fdr} below. We see that first our procedure is relatively robust to
the choice of $\kappa$ in both FDR control and power across different correlations $\rho$. Secondly,
for all choices of $\kappa$, our procedure controls the FDR at $0.2$ across all the correlation coefficients $\rho$.
Moreover, it enjoys an approximately $20 \%$ increase in power 
compared with Figure 1 in \cite{barber2016knockoff}. 

\begin{figure}
  \caption{FDR control and power for square-root Lasso across different correlations $\rho$.} 
  \label{fig:fdr}
    \includegraphics[width=\textwidth]{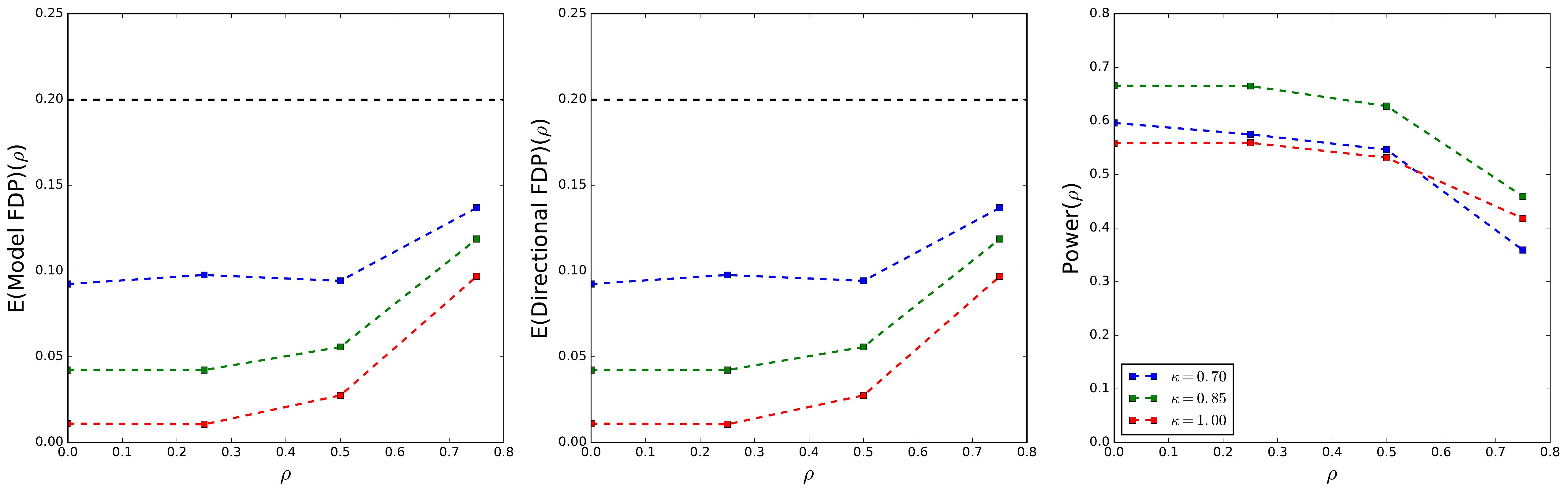}
\end{figure}

\bibliographystyle{agsm}
\bibliography{paper}

\end{document}